\documentclass[english,11pt]{amsart}
\usepackage{amsopn,amsthm,amsfonts,amsmath,amssymb,amscd}
\usepackage{babel}
\usepackage{amssymb,tikz}

\newtheorem{Theorem}{Theorem}[section]

\newtheorem{Corollary}[Theorem]{Corollary}
\newtheorem{Example}[Theorem]{Example}
\newtheorem{Remark}[Theorem]{Remark}

\newenvironment{Proof*}{{\it Proof.}}

\newcommand{\NN}{\mathbb{N}}
\newcommand{\BB}{\mathbb{B}}
\newcommand{\Nil}{\mathcal{N}}

\newcommand{\ZZ}{\mathbb{Z}}

\begin{document}

\title{Semirings generated by idempotents}

\author{David Dol\v zan}

\address{D.~Dol\v zan:~Department of Mathematics, Faculty of Mathematics
and Physics, University of Ljubljana, Jadranska 19, SI-1000 Ljubljana, Slovenia, and Institute of Mathematics, Physics and Mechanics, Jadranska 19, SI-1000 Ljubljana, Slovenia; e-mail: 
david.dolzan@fmf.uni-lj.si}
\bigskip

\begin{abstract} 
 We prove that a semiring multiplicatively generated by its idempotents is commutative and Boolean, if every idempotent in the semiring has an orthogonal complement. We prove that a semiring additively generated by its idempotents is commutative, if every idempotent in the semiring has an orthogonal complement and all the nilpotents in the semirings are central. We also provide examples that the assumptions on the existence of orthogonal complements of idempotents and the centrality of nilpotents cannot be omitted.
\end{abstract}

\subjclass[2010]{} 
\thanks{The author acknowledges the financial support from the Slovenian Research Agency  (research core funding No. P1-0222)}

\maketitle 

 \section{Introduction}

\bigskip

The study of idempotents has always been an important tool in algebra and in particular, it has proved to be useful in studying Artinian and finite rings. Idempotents dominate the structure theory of rings. In 1966, Steger \cite{steger} checked the conditions of diagonability of idempotent matrices over commutative rings with identity. Much of the work has been done to examine the properties of idempotents in connection with the structure of rings, see for example \cite{camillo, dube, nicholson}. It is well known that idempotents induce direct sum decompositions of rings (Peirce decompositions) which in turn determine the structure of rings, provided that the rings have enough idempotents. For some recent results in this area see \cite{anh} and \cite{anh1}. Much research has also been devoted to the properties of rings with either small or large number of idempotents - see \cite{chen, chin}. Since the exact number of idempotents in a ring can be difficult to obtain in general, some authors have tried to establish the number of idempotents in some special classes of rings (\cite{kanwar, mittal, weiss}), while others have tried to find bounds for the number of idempotents in a ring (\cite{cheraghpour, dolzan21, machale}).

A \emph{semiring} is a set $S$ equipped with binary operations $+$ and $\cdot$ such that $(S,+)$ is a commutative monoid 
with identity element 0, and $(S,\cdot)$ is a monoid with identity element 1. In addition, operations $+$ and $\cdot$
are connected by distributivity and 0 annihilates $S$. A semiring is \emph{commutative} if $ab=ba$ for all $a,b \in S$.
The theory of semirings has many applications in optimization theory, automatic control, models of discrete event networks and graph theory (see e.g. \cite{baccelli, cuninghame, li2014, zhao2010}).  For an extensive theory of semirings, we refer the reader to \cite{hebisch}. There are many natural examples of commutative semirings, for example, the set of nonnegative integers (or reals) with the usual operations of addition and multiplication. Other examples include distributive lattices, tropical semirings, dio\"{\i}ds, fuzzy algebras, inclines and bottleneck algebras. 

A lot of study has been devoted to properties of rings and algebras that are multiplicatively generated by idempotents (\cite{chuang, putcha}) or idempotents and nilpotents (\cite{abu}). Recently, the authors in \cite{chen2021} have also studied rings, which are additively generated by idempotents and nilpotents. As far as the author of this paper is aware, these properties have not yet been studied in the semiring setting. 

This is therefore the aim of this paper, which is organized as follows. The next section contains the necessary definitions and preliminaries. In Section 3, we study semirings generated by idempotents. Theorem 1 in \cite{putcha} states that every ring with identity that is multiplicatively generated by its idempotents is Boolean (and thus also commutative). We show (see Example \ref{notinsemi}) that this is not true in the semiring setting. We do however, prove that any semiring that is multiplicatively generated by idempotents with the additional property that for every idempotent $e \in S$ there exists an orthogonal complement (that is, an idempotent $f \in S$ such that $e+f=1$ and $ef=fe=0$), is Boolean and commutative (see Theorem \ref{main}). Furthermore, we loosen these additional assumptions somewhat in Corollary \ref{main2}, where we examine the case where  for every idempotent $e \in S$ there exist a nilidempotent $f$ and a nilpotent $x$ such that $e+f=1+x$, and we prove that the semiring is Boolean and commutative in this case as well.
We then examine semirings multiplicatively generated by nilidempotents (see Theorem \ref{mainnilid}) and prove that under certain conditions on the nilpotent elements, the semiring has to be commutative (but not necessarily Boolean, as Example \ref{notboolean} shows).
Next, we examine semirings that are additively generated by idempotents. We prove (see Theorem \ref{additivecom}) that every semiring that is additively generated by idempotents, where every idempotent has an orthogonal complement and all the nilpotents are contained in the centre, is commutative and give examples that neither of these last two assumptions can be omitted (see Example \ref{opomba1}).
Finally, we provide two examples (Example \ref{additively}) that semirings that are additively generated by idempotents are in general not Boolean, even under all the assumptions of Theorem \ref{additivecom}.

\bigskip

\section{Definitions and preliminaries}

\bigskip

The simplest example of a semiring  is the binary Boolean semiring, the set $\{0,1\}$ in which addition and multiplication are the same as in $\ZZ$ except that $1+1=1$. We shall denote the binary Boolean semiring by $\BB$.


An element $e \in S$ is called \emph{idempotent} if $e^2=e$. An idempotent $e$ is nontrivial if $e \notin \{0,1\}$. An element $x \in S$ is called \emph{nilpotent} if there exists an integer $k$ such that $x^k=0$. Furthermore, $e$ is called \emph{nilidempotent} if $e^2=e+x$ for some nilpotent element $x$.

Let us  denote the set of all multiplicatively invertible elements in $S$ by $S^*$, the set of all additively invertible elements in $S$ by $V(S)$ and the set of all idempotent elements in $S$ by $I(S)$.
Furthermore, let $\Nil(S)$ denote the set of all nilpotent elements in $S$ and let $Z(S)$ denote the centre of $S$. 
We say that semiring $S$ is \emph{Boolean} if $I(S)=S$.

 If $e$ and $f$ are idempotents, we say that $e$ and $f$ are \emph{mutually orthogonal} if $ef=fe=0$. If $e$ and $f$ are nilidempotents, we say that $e$ and $f$ are \emph{mutually nilorthogonal} if $ef, fe \in \Nil(S)$.
Choose $b \in S$. A set $\{a_1,a_2,\ldots,a_r\} \subseteq S$ of nonzero mutually orthogonal idempotents is called an \emph{orthogonal decomposition of $b$, (of length $r$)} if $a_1+a_2+\ldots+a_r=b$ and similarly a set $\{a_1,a_2,\ldots,a_r\} \subseteq S$ of nonzero mutually nilorthogonal nilidempotents is called an \emph{nilorthogonal decomposition of b, (of length $r$)} if $a_1+a_2+\ldots+a_r=b$.  Furthermore, we say that the idempotent $e$ has an \emph{orthogonal complement} if there exists an idempotent $f$ such that $\{e, f\}$ is an orthogonal decomposition of $1$; and the nilidempotent $e$ has a \emph{nilorthogonal complement} if there exist a nilidempotent $f$ and a nilpotent $x$ such that $\{e, f\}$ is a nilorthogonal decomposition of $1+x$.

For a semiring $S$, we shall denote the set of all $n$ by $n$ matrices with entries in $S$ by $M_n(S)$ and the set of all $n$ by $n$ upper triangular matrices with entries in $S$ by $T_n(S)$. We shall denote the identity matrix in $M_n(S)$ by $I$.

\bigskip

\section{Semirings generated by idempotents}

\bigskip

This is the main section of this paper, where we study the semirings that are either multiplicatively or additively generated by idempotents. Firstly, we turn our attention to the semirings that are multiplicatively generated by idempotents. 

As the following example shows, a semiring can be multiplicatively generated by its idempotents and yet be neither commutative nor Boolean, in contrast to the situation in the ring setting (compare \cite[Theorem 1]{putcha}).

\bigskip

\begin{Example}
\label{notinsemi}
Let $S=T_2(\BB)$ where $\BB$ is the binary Boolean semiring. Denote $A =  \left[
 \begin{matrix}
 0 & 1 \\
 0 & 0 
 \end{matrix}
 \right] \in S$. 
Note that $I(S)=S \setminus \{ A \}$. Observe that $A = \left[
 \begin{matrix}
 1 & 0 \\
 0 & 0 
 \end{matrix}
 \right]  \left[
 \begin{matrix}
 0 & 1 \\
 0 & 1 
 \end{matrix}
 \right] $, so $S$ is multiplicatively generated by its idempotents. However, $S$ is neither commutative nor Boolean.
\end{Example}

\bigskip

This encourages us to look for some additional property of a semiring, such that the fact that it is multiplicatively generated by its idempotents would then imply that it is Boolean and commutative. Hence, we prove the next theorem. 

\bigskip

\begin{Theorem}
\label{main}
Let $S$ be a semiring multiplicatively generated by idempotents and suppose that every idempotent in $S$ has an orthogonal complement. Then $S$ is commutative and Boolean.
\end{Theorem}
\begin{proof}
Choose $x \in S$ and $e \in I(S)$. By the assumption, there exists $f \in I(S)$ such that $e+f=1$ and $ef=fe=0$.
Observe that $(e+exf)^2=e+exf$, so $e+exf \in I(S)$ and there exists $g \in I(S)$ such that 
\begin{equation}
\label{eq1}
e+exf+g=1
\end{equation} 
and $(e+exf)g=g(e+exf)=0$.
Note that multiplying equation (\ref{eq1}) with $f$ from the left side yields $fg=f$.

Now denote $t=(1+exf)(e+g)=e+g+exfg$. Since $eg+exfg=0$, we have $et=e$. Also, $fg=f$ yields $ft=f$. This implies that $t=(e+f)t=e+f=1$, so we have proved that
$1+exf$ has a right inverse in $S$. 

Since $S$ is multiplicatively generated by idempotents, we can write $1+exf=h_1 h_2 \ldots h_k$ for some $h_1, h_2, \ldots, h_k \in I(S)$. Thus $h_1(1+exf)=1+exf$ and since $1+exf$ has a right inverse, this implies that $h_1=1$. Inductively, we can then prove that $h_2=\ldots=h_k=1$ as well.
So, we can conclude that $1+exf=1$ and multiplying this with $e$ from the left and $f$ from the right we ge $exf=0$. This further implies that $ex=ex(e+f)=exe$.

By a symmetrical argument, we can prove also that $xe=exe$, so $ex=xe$. We have therefore proved that every idempotent in $S$ lies in the centre $Z(S)$. Since $S$ is generated by idempotents, this implies that $S$ is commutative and therefore also Boolean.
\end{proof}

\bigskip

\begin{Remark}
Observe that in the semiring from Example \ref{notinsemi}, one can check that the idempotent $\left[
 \begin{matrix}
 1 & 1 \\
 0 & 0 
 \end{matrix}
 \right]$ does not have an orthogonal complement. It does however, have a nilorthogonal complement, since
$\left[
 \begin{matrix}
 0 & 1 \\
 0 & 0 
 \end{matrix}
 \right]$ is a nilpotent and 
$\left[
 \begin{matrix}
 1 & 1 \\
 0 & 0 
 \end{matrix}
 \right] + \left[
 \begin{matrix}
 0 & 1 \\
 0 & 1 
 \end{matrix}
 \right] = 1 + \left[
 \begin{matrix}
 0 & 1 \\
 0 & 0 
 \end{matrix}
 \right]$ and 
$\left[ \begin{matrix}
 1 & 1 \\
 0 & 0 
 \end{matrix}
 \right]  \left[
 \begin{matrix}
 0 & 1 \\
 0 & 1 
 \end{matrix} \right], \left[ \begin{matrix}
 0 & 1 \\
 0 & 1 
 \end{matrix}
 \right] \left[
 \begin{matrix}
 1 & 1 \\
 0 & 0 
 \end{matrix}\right] \in \Nil(S)$.
\end{Remark}

\bigskip 

Let us therefore examine the semirings such that every idempotent has a nilorthogonal complement. Obviously, as Example \ref{notinsemi} shows,
this still does not suffice for the semiring that is multiplicatively generated by its idempotents to be commutative and Boolean. However, if we assume some additional properties on the set of nilpotents, we do manage to get the desired results. More precisely, we have the following corollary.
One part of its proof could be modified from the proof of Theorem 2.5 from \cite{dolzaninvert}, but the assumptions that we have here are slightly different, so we include all the details for the sake of completeness.


\bigskip

\begin{Corollary}
\label{main2}
Let $S$ be a semiring multiplicatively generated by idempotents and suppose that every idempotent in $S$ has a nilorthogonal complement. 
If $\Nil(S) \subseteq V(S) \cap Z(S)$, then $S$ is commutative and Boolean.
\end{Corollary}
\begin{proof}
Choose $e \in I(S)$. By the assumption, there exist a nilidempotent $g$ and $y \in \Nil(S)$ such that $e+g=1+y$ and $eg,ge\in \Nil(S)$. 
We firstly prove that there exists an idempotent $f$ such that $f=g+n$ for some $n \in \Nil(S)$. 
Since $g$ is a nilidempotent, we have $g^2=g+z$ for some $z \in \Nil(S)$. If $z=0$ then $f=g$ is an idempotent and we are done. Suppose therefore that 
$z \neq 0$. Denote $z_1=z$, $g_1=g$ and $k=1$. We now proceed inductively on $k$. We have $g_k^2=g_k+z_k$. Since $z_k \in V(S)$, we can define $w_{k+1}=z_k+2g_k(-z_k) \in \Nil(S)$ and $g_{k+1}=g_k+w_{k+1}$. Since $w_{k+1} \in Z(S)$, we have that $g_{k+1}^2=g_k^2+2g_kw_{k+1}+w_{k+1}^2$ and since $z_k \in Z(S)$ we further have $g_{k+1}^2=g_k+z_k+2g_kz_k+4(g_k+z_k)(-z_k)+z_k^2+4g_k(-z_k^2)+4(g_k+z_k)z_k^2=g_{k+1}+4z_k^3+3(-z_k^2)$. Denote $z_{k+1}=4z_k^3+3(-z_k^2)=-z_k^2(3-4z_k) \in \Nil(S)$, so $g_{k+1}^2=g_{k+1}+z_{k+1}$. Observe also that $z_{k+1}=z^{2^k}s$ for some $s \in S$. Since $z \in \Nil(S)$, there exists $k \in \NN$ such that 
$z_{k+1}=0$ and therefore $g_{k+1}$ is an idempotent. So, $f=g_{k+1}=g+\sum_{t=2}^{k+1}{w_t}$ and $\sum_{t=2}^{k+1}{w_t} \in \Nil(S)$. 

Observe that $eg, ge \in \Nil(S)$ also implies $ef, fe \in \Nil(S)$. 
The fact that $\Nil(S) \subseteq V(S)$ further implies that $e+f=1+x$ for some $x \in \Nil(S)$ with $ef,fe \in \Nil(S)$.
Since $\Nil(S) \subseteq Z(S)$, this implies $ef=e(ef)f=(ef)^2$, so $ef=0$. Similarly we prove $fe=0$. Observe that this implies that $(1+x)^2=1+x$.
Since $x \in \Nil(S)$, there exists $n \in \NN$ such that $x^n=0$ and since $\Nil(S) \subseteq V(S)$, we have $(1+x)(1-x)(1+x^2)\ldots(1+x^{2^{k-1}})=1-x^{2^{k}}=1$ for $k \geq \log_2(n)$, so $1+x \in S^*$. This now gives us $1+x=1$, so $e+f=1$, so we have proved that $e$ has an orthogonal complement.
The result now follows from Theorem \ref{main}.
\end{proof}

\bigskip

We now loosen the condition that the semiring is multiplicatively generated by idempotents and assume that the semiring is multiplicatively generated by nilidempotents. We have the following theorem.

\bigskip

\begin{Theorem}
\label{mainnilid}
Let $S$ be a semiring multiplicatively generated by nilidempotents and suppose that every idempotent in $S$ has a nilorthogonal complement. 
If $\Nil(S) \subseteq V(S) \cap Z(S)$, then $S$ is commutative.
\end{Theorem}
\begin{proof}
In the same way as in the proof of Corollary \ref{main2}, we prove that every idempotent has an orthogonal complement. This means that we can proceed in the same way as in the first part of the proof of Theorem \ref{main} to prove that $1+exf \in S^*$ for every $e \in I(S)$ and every $x \in S$ (where $f$ denotes the orthogonal complement of $e$). The assumption that $S$ is multiplicatively generated by nilidempotents now yields $1+exf=(h_1+n_1)(h_2+n_2) \ldots (h_k+n_k)$ for some idempotents $h_1,h_2,\ldots,h_k$ and some nilpotents $n_1, n_2, \ldots, n_k$. Since $\Nil(S) \subseteq Z(S)$, this implies $1+exf=h_1h_2 \ldots h_k + n$ for some $n \in \Nil(S)$. Denote $u=h_1h_2 \ldots h_k$ and observe that we can write $u=1+exf-n$ since $n \in V(S)$. Furthermore, denote $v=1+exf \in S^*$ and $n' = v^{-1}n \in \Nil(S)$ and note that $u=v(1+n')$. Since $n' \in \Nil(S)$, we have $n'^r=0$ for some integer $r$ and the fact that $n' \in V(S)$ now yields $1+n' \in S^*$. Therefore, we have proved that $u  \in S^*$. Now, we can proceed similarly as in the last part of the proof of Theorem \ref{main} to prove that $u=1+exf-n=1$. Multiplying this with $e$ from the left and $f$ from the right side, together with the fact that $\Nil(S) \subseteq Z(S)$ gives us $exf=0$ and consequently $ex=exe$. Similarly we prove that $xe=exe$, so $I(S) \subseteq Z(S)$. Since $S$ is multiplicatively generated by nilidempotents, we have proved that $S$ is commutative. 
\end{proof}

\bigskip

\begin{Example}
\label{notboolean}
It does not necessarily follow that the semiring $S$ from Theorem \ref{mainnilid} is also Boolean. Consider two examples:
\begin{enumerate}
\item
ring $S=\ZZ_2[x]/(x^2)$, and
\item semiring $\BB[x,y]$, where all its elements are additively idempotent and $x+y=xy=yx=x^2=y^2=0$. 
\end{enumerate}
Both examples can be seen to satisfy all the conditions of Theorem \ref{mainnilid}, but none of the (semi)rings in question is Boolean.
\end{Example}

\bigskip

\begin{Remark}
\label{opomba}
In all the above cases (Theorem \ref{main}, Corollary \ref{main2} and Theorem \ref{mainnilid}), we have proved that the semiring in question is commutative. Since we have proved that every idempotent also has an orthogonal complement, this implies that in the case that $S$ is a finite semiring, we can use the Peirce decomposition to decompose the semiring to the direct product of semirings, $S \simeq S_1 \times S_2 \times \ldots S_k$ where each of the semirings $S_i$ has no nontrivial idempotents. In the case of Theorem \ref{main} and Corollary \ref{main2}, this further implies that either $S_i \simeq \BB$ or $S_i \simeq \ZZ_2$ for every $i$.
\end{Remark}

\bigskip

Next, let us briefly examine semirings that are additively generated by their idempotents.

\bigskip

\begin{Theorem}
\label{additivecom}
Let $S$ be a semiring additively generated by idempotents and suppose that every idempotent in $S$ has an orthogonal complement. 
If $\Nil(S) \subseteq Z(S)$, then $S$ is commutative.
\end{Theorem}
\begin{proof}
Suppose $e \in I(S)$ and choose $x \in S$. There exists $f \in I(S)$ such that $e+f=1$ and $ef=fe=0$. But since $exf \in \Nil(S) \subseteq Z(S)$, we have
$exf=0$. This implies that $ex=exe$ and similarly we prove $xe=exe$, therefore $e \in Z(S)$. Since $S$ is additively generated by idempotents, $S$ is commutative. 
\end{proof}

\bigskip

\begin{Example}
\label{opomba1}
\begin{enumerate}
\item
Obviously, the assumption $\Nil(S) \subseteq Z(S)$ is crucial here. It can be readily verified that the non-commutative ring $M_2(\ZZ_2)$ is additively generated by idempotents.  

\item
Also, the assumption that every idempotent in $S$ has an orthogonal complement cannot be omitted. Consider for example the semiring of polynomials with non-negative integer coefficients over two non-commuting variables, $S=\NN\langle x,y \rangle=\{a+bx+cy; a,b,c \in \NN\}$, where $x^2=x,y^2=y,xy=x$ and $yx=y$. It can be easily verified that $S$ is a non-commutative semiring that is additively generated by idempotents and that $\Nil(S)=\{0\} \subseteq Z(S)$.
\end{enumerate}
\end{Example}

\bigskip

Finally, we give some examples that in case semiring $S$ is additively generated by
 its idempotents, we cannot draw the conclusion that $S$ is Boolean, even when all the other assumptions from Theorem \ref{additivecom} are satisfied.

\bigskip

\begin{Example}
\label{additively}
\begin{enumerate}
\item
The semiring $\NN$ of all nonnegative integers is additively generated by its idempotents and satisfies all the conditions of Theorem \ref{additivecom}, but $\NN$ is not Boolean.
\item
The same can be said in the case $S=\ZZ_3[x]/(x^2-1)$. Note that here, $S$ is even a finite commutative ring that is additively generated by idempotents (and it also contains nontrivial idempotents).
\end{enumerate}
\end{Example}

\end{document}